\documentclass[preprint,12pt]{elsarticle}

\usepackage{ amsmath , amssymb,latexsym}
\usepackage{amsthm}
\usepackage{hyperref}
\hypersetup{pdftex,colorlinks=true,allcolors=blue}
\usepackage{hypcap}
\usepackage[all]{xy}
\xyoption{all}
\usepackage{graphicx}
\usepackage[font=small,labelfont=bf]{caption} 

\newcommand{\nc}{\newcommand}
\nc{\Aut}{{\mbox{Aut}}}
\nc{\Deck}{{\mbox{Deck}}}
\DeclareMathOperator{\deck}{Deck}
\nc{\bC}{{\mathbb{C}}}
\nc{\bP}{{\mathbb{P}}}
\nc{\bZ}{{\mathbb{Z}}}
\newtheorem{theorem}{Theorem}
\newtheorem{lemma}{Lemma}
\newtheorem{remark}{Remark}
\newtheorem{corollary}{Corollary}

\journal{Pure and Applied Algebra}

\begin{document}

\begin{frontmatter}



\title{Periods of generalized Fermat curves}


\author[1]{Yerko Torres-Nova}
\ead{yerko.torresn@gmail.com}

\address{Departamento de Matem\'aticas, Pontificia Universidad Cat\'olica de Chile, Santiago, Chile}

\begin{abstract}
Let $k,n \geq 2$ be integers. A generalized Fermat curve of type $(k,n)$ is a compact Riemann surface $S$ that admits a subgroup of conformal automorphisms $H \leq \Aut(S)$ isomorphic to $\mathbb{Z}_k^n$, such that the quotient surface $S/H$ is biholomorphic to the Riemann sphere $\hat{\mathbb{C}}$ and has $n+1$ branch points, each one of order $k$. There exists a good algebraic model for these objects, which makes them easier to study. Using tools from algebraic topology and integration theory on Riemann surfaces, we find a set of generators for the first homology group of a generalized Fermat curve. Finally, with this information, we find a set of generators for the period lattice of the associated Jacobian variety.

\end{abstract}


\end{frontmatter}

\section{Introduction}
\label{intro}
The Jacobian variety $JS$ of a compact Riemann surface $S$ of genus $g$ is isomorphic to a complex torus of dimension $g$, i.e., a quotient $\mathbb{C}^g/ \Lambda$, where $\Lambda  \subset \mathbb{C}^g$ is the period lattice ( $\Lambda \cong \mathbb{Z}^{2g}$ ) of $S$ that depends on the analytical and algebraic-topological structure of $S$.
The importance of $JS$ is due to Torelli's theorem, which states that the principally polarized abelian variety $JS$ determines the Riemann surface $S$ up to biholomorphism.

Thus, if the Jacobian variety is in the form $\mathbb{C}^g/ \Lambda$, the period lattice $\Lambda$ with the corresponding polarization determines $S$.
However, to find an explicit form for the period lattice of a particular compact Riemann surface is a difficult task and there is no  standard method to do it. 

We restrict attention to an interesting family of compact Riemann surfaces called generalized Fermat curves of type $(k,n)$, where $k,n \geq 2$ are integers. In \cite{GHL} it was noticed that such a Riemann surface can be described as a suitable fiber product of $(n-1)$ classical Fermat curves of degree $k$. 
In this paper we find a generating set for the period lattice of a generalized Fermat curve, based on the work of Rohrlich \cite{GR} who found a generating set for the period lattice of the classical Fermat curve of degree $k \geq 4$.

\section{Preliminaries}

\label{sec:1}
\subsection{The Jacobian variety}
Let $S$ be a compact Riemann surface of genus $g\geq 0$. Its first homology group $H_1(S , \bZ)$ is a free Abelian group of rank $2g$, and the complex vector space $H^{1,0}(S)$ of its holomorphic $1$-forms has dimension $g$. There is a natural $\bZ$-linear injective map
$$\tau : H_1(S, \bZ) \hookrightarrow (H^{1,0}(S))^{*} $$
$$\gamma \mapsto \tau(\gamma)(\cdot) : = \int_{\gamma} \cdot,  $$

\noindent where  $(H^{1,0}(S))^{*}$ is the dual space of  $H^{1,0}(S)$. The image $\tau(H_1(S, \bZ))$ is a lattice in $(H^{1,0}(S))^{*}$, and the quotient $g$-dimensional torus
$$JS := (H^{1,0}(S))^{*}/\tau(H_1(S, \bZ))  $$
is called the \emph{Jacobian variety} of $S$. It is a fact that $JS$ admits a principal polarization defined by the Hermitian form on $H^{1,0}(S)$ given by
$$(\omega_1 , \omega_2) \to \int \omega_1 \wedge \overline{\omega_2} .$$

If $\{ \omega_1 , \dotso, \omega_g\}$ is a basis for $H^{1,0}(S)$, then we have the isomorphism $(H^{1,0}(S))^{*} \cong \bC^g$, and if $ \{ \gamma_1,...,\gamma_m\} $ is a finite generating set for $H_1(S, \bZ)$ (need not be a basis), then we can see $\tau(H_1(S, \bZ))$ as the lattice $\Lambda$ in $\bC^g$  generated by the collection
$$C_i = \left(\int_{\gamma_i} \omega_1, \int_{\gamma_i} \omega_2, \dotso, \int_{\gamma_i} \omega_g  \right),\hspace{0.2cm} 1 \leq i \leq m .$$
\noindent The lattice generated by the  $C_i `s$ is called the \emph{period lattice} of $S$, and in the case where $m$ is the rank $2g$ of $H_1(S , \bZ)$, we can find the Riemann matrix of $S$, which allows us to study $JS$ as a polarized Abelian variety.

\subsection{ Generalized Fermat curves}
Let $k,n \geq 2$ be integers. A compact Riemann surface $S$ is called a \emph{generalized Fermat curve} of type $(k,n)$ if it admits a subgroup of conformal automorphisms $H \leq \Aut(S)$ that is isomorphic to $\mathbb{Z}_k^n$ (where $\bZ_k = \bZ/k\bZ$), such that the quotient surface $S/H$ is biholomorphic to the Riemann  sphere $\hat{\bC}$ and has $n+1$ branch points, each one of order $k$. In this case the subgroup $H$ is called a \emph{generalized Fermat group} of type $(k,n)$, and the pair $(S,H)$ is called a \emph{generalized Fermat pair of type $(k,n)$}. As a consequence of the Riemann-Hurwitz formula given in Corollary 1.2 of \cite{KZ1} or Proposition 1.2 of \cite{KZ2}, the genus $g_{k,n}$ of a generalized Fermat curve of type $(k,n)$ is
$$g_{k,n} = \frac{2 + k^{n-1}((n-1)(k-1) -2)}{2}.$$

We say that two generalized Fermat pairs $(S_1 , H_1)$ and $(S_2,H_2)$ are \emph{holomorphically equivalent} if there exists a biholomorphism $f:S_1 \to S_2$ such that $fH_1 f^{-1} = H_2$. \\

\begin{remark}
Note that generalized Fermat curves of type $(k,1)$ are just cyclic covers of degree $k$ of $\hat{\bC}$ with two branch points, which are all of genus $0$.
\noindent From \cite{TZ} we know that the Fermat curve of degree $k\geq 2$ given by
 $$\{[x_0,x_1,x_2] \in \bP^2\bC : x_0^k + x_1^k + x_2^k=0\}$$ has a subgroup of conformal automorphisms isomorphic to $\bZ_k^2$, where the quotient surface is biholomorphic to the Riemann sphere with three branch points $\infty,0,1$. Thus the classical Fermat curves are generalized Fermat curves of type $(k,2)$.
\end{remark}

\begin{remark}
The non-hyperbolic case, i.e., when $g_{k,n} \leq 1$, are given by $(k,n) \in \{(2,2) , (2,3) , (3,2)\}$, or $k=1$. See \cite{GHL} for explicit examples. \\
\end{remark}

Let $(S,H)$ be a generalized Fermat pair of type $(k,n)$ and, up to a Moebius transformation, let $\{\infty, 0,1,\lambda_1, \lambda_2,..., \lambda_{n-2} \}$ be the branch points of the quotient $S/H$. Let us consider the following fiber product of $n-1$ classical Fermat curves:
\begin{equation} \label{eq:typekn}
C_k(\lambda_1,..., \lambda_{n-2}) := 
\left\{\begin{array}{ccc} x_0^k + x_1^k + x_2^k & = &0 \\ \lambda_1 x_0^k + x_1^k + x_3^k & = &0 \\ \lambda_2 x_0^k + x_1^k + x_4^k & = &0
\\ \vdots & \vdots & \vdots \\ \lambda_{n-2} x_0^k + x_1^k + x_{n}^k & = &0  \end{array} \right\} \subset \bP^n \bC.
\end{equation}

Since the values $\lambda_i$ are pairwise different and each one is different from $0$ and $1$, the algebraic curve $C_k(\lambda_1,...,\lambda_{n-2})$ is a non-singular projective algebraic curve, hence a compact Riemann surface.\\

On $C_k(\lambda_1,...,\lambda_{n-2})$ we have the abelian group $H_0 \cong \bZ_k^n$ of conformal automorphisms generated by the maps
$$a_i([x_0, \cdots, x_n]) = [x_0, \cdots , x_{i-1} ,\zeta_k x_i , x_{i+1}, \cdots ,x_n],\quad i=0,...,n, $$
where $\zeta_k = e^{2\pi i/k }$. Let us consider the holomorphic map of degree $k^n$
$$\begin{array}{cccc}\pi : & C_k(\lambda_1,...,\lambda_{n-2})& \to &\hat{\bC} \\ & [x_0, \cdots , x_n] & \mapsto & - \left(\frac{x_1}{x_0}\right)^k,\end{array} $$

\noindent with the property $\pi \circ a_i = \pi$ for each $i=1,\dotso,n$. So $\pi$ induces a biholomorphism

$$\begin{array}{cccc}\hat{\pi}: & C_k(\lambda_1,...,\lambda_{n-2})/H_0& \to &\hat{\bC} \\ &H_0 p& \mapsto & \pi(p).\end{array} $$

\noindent  Furthermore, the map $\pi$ has $n+1$ branch points given by  $$\{ \infty, 0,1, \lambda_1,\lambda_2,...,\lambda_{n-2} \}.$$ 
It follows that $C_k(\lambda_1,...,\lambda_{n-2})$ is a generalized Fermat curve of type $(k,n)$ with generalized Fermat group $H_0$, whose standard generators are $a_1, \dotso, a_n$ and $a_{0} = (a_1a_2\dotso a_{n})^{-1}$.  Using the above notation, the following result was proved in \cite{GHL}.

\begin{theorem}
The generalized Fermat pairs $(S,H)$ and $(C_k(\lambda_1,...,\lambda_{n-2}) , H_0)$ are holomorphically equivalent.
\end{theorem}

On $C_k(\lambda_1,...,\lambda_{n-2})$ we have the following meromorphic maps
$$y_j = \frac{x_{j}}{x_0} : C_k(\lambda_1,...,\lambda_{n-2}) \to \hat{\bC},\hspace{0.2cm} j = 1,...,n.$$
We consider the set $I_{k,n}$ of tuples $(\alpha_1, \dotso, \alpha_{n})$ such that
 $$ \alpha_i \in \bZ,\quad 0\leq \alpha_2,\dotso, \alpha_{n} \leq k-1,\quad 0 \leq \alpha_1 \leq \sum_{i=2}^{n}\alpha_i -2, $$
and define the meromorphic form
$$\theta_{\alpha_1,...,\alpha_n} := \frac{y_1^{\alpha_1} dy_1}{y_2^{\alpha_2} \dotso y_{n}^{\alpha_{n}}},$$
\noindent for each $(\alpha_1  , \dotso, \alpha_{n}) \in I_{k,n}$.The paper \cite{RH} proved the following.

\begin{theorem} With the above notation, the following holds:
\begin{enumerate}
\item $\theta_{\alpha_1,...,\alpha_{n}}$ is holomorphic for every $(\alpha_1,\dotso, \alpha_{n}) \in I_{k,n}$.
\item  $\# I_{k,n} = g_{k,n}$.
\item The collection $$\{ \theta_{\alpha_1,...,\alpha_{n}} \}_{(\alpha_1,...,\alpha_{n} )\in I_{k,n}}$$ is a basis for the space $H^{1,0}(C_k(\lambda_1,...,\lambda_{n-2}))$ of holomorphic $1$-forms.
\end{enumerate}
\end{theorem}

For simplicity, in the rest of this paper we write $C_{k,n}$ instead of $C_k(\lambda_1,...,\lambda_{n-2})$.

\subsection{The logarithm symbol on the punctured plane} \label{logarith}

Let $R \subset \bC$ a finite subset with $0 \in R$ and $|R|\geq 2$. The elements of $R$ are denoted by $r_i$, with $1\leq i \leq |R|$. Then we consider a universal covering of $\bC -R$ given by
$$ p: U \to \bC -R. $$ 

Since $p$ is holomorphic, we have the family of holomorphic functions $p_i = p - r_i$ with $1\leq i \leq |R|$. The function $p_i$ does not vanish on $U$, so there exists a determination of $\log p_i$ on $U$ such that 
$$\exp(\log p_i) = p_i. $$ 

Let $\deck(p)$ be the group of covering transformations of $p$. Then for every $\phi \in \deck(p)$ the function 
$$u \to \frac{1}{2\pi i} (\log p_i (\phi(u)) - \log p_i(u))$$
on $U$ is identically an integer. This integer is independent of the choice of $\log p_i$, and we denote it by $L(p_i , \phi)$. It is not difficult to see that the symbol $L(p_i, \cdot)$  satisfies

\begin{equation}\label{eq:homo} 
L(p_i , \phi \circ \psi) = L(p_i , \phi) + L(p_i , \psi),
\end{equation}

\noindent for every $\phi,\psi \in \deck(p)$. Furthermore, we observe the following.

\begin{lemma}\label{powerlog}
Let $\hat{x}_i:U \to \bC$ be the $k$th root of $p_i$ defined by
$$  \hat{x}_i =\exp\left(\frac{1}{k} \log p_i \right) .$$  
Then for every $\phi \in \deck(p)$ we have
$$\hat{x}_i \circ \phi = \zeta_k^{L(p_i , \phi)} \hat{x}_i .$$
\end{lemma}

Recall that $\deck(p)$ is isomorphic to the fundamental group $\pi_1(\bC - R)$, which is a free group generated by $|R|$ elements, each one homotopic to a circle with center $r_i$ and index one. Then we consider the generators $\phi_1,...,\phi_{|R|} \in \deck(p)$ associated with each generator of $\pi_1(\bC - R)$, and  for any $u\in U$ we have the equality
$$L(p_i , \phi_j) = \frac{1}{2\pi i} \int_{u}^{\phi_j(u)} d\log p_i  = \delta_{ij},$$  

\noindent where $\delta_{ij}$ is the usual Kronecker delta. \\

For general aspects of the logarithm symbol on Riemann surfaces, see \cite{LA}.

\label{sec:2}
\section{Generating set for the period lattice of $C_{k,n}$}

Consider the generalized Fermat curve of type $(k,n)$ given by Equation \eqref{eq:typekn} and the set of $n+1$ branch points $R \cup \{\infty \}$, where $$R = \{ r_1 = 0, r_2 = \lambda_0 = 1, r_3=\lambda_1, \dotso ,r_n=\lambda_{n-2}\}.$$

\subsection{A finite generating set for $H_1(C_{k,n}, \bZ)$}\label{homology}
 
Associated with $R$, we have the universal covering $ p: U \to \bC -R $. We have the set of functions 
$$
\left\{ \begin{array}{cccc} p_1 &= & -p& \\ p_i & = & p -r_i& \text{ for } 2\leq i\leq n. \end{array} \right.$$
There exists a $k$th root $\hat{x}_i$ of $p_i$, which by Lemma \ref{powerlog} satisfies
$$\hat{x}_i \circ \phi = \zeta_k^{L(p_i , \phi)} \hat{x}_i  $$
for each $\phi \in \deck(p)$. From Equation \eqref{eq:homo} we have the surjective homomorphism
$$\Psi: \deck(p) \to \bZ_k^n , \quad \Psi(\phi )= ( L(p_1,\phi)  , \dotso, L(p_{n},\phi) ) \mod k , $$ 
so it is not difficult to deduce the following fact.
\begin{lemma}
The subgroup of $\deck(p)$ which leaves each $\hat{x}_i$ invariant  is
$$\deck(p)_k:= \mbox{Ker} (\Psi). $$
\end{lemma}

Now we consider the punctured Riemann surface $C_{k.n}' = C_{k,n} - \pi^{-1}(R\cup\{\infty\})$. We now prove the following result.

\begin{lemma}\label{deckk}
The map
$$q:U \to C_{k,n}', \quad u \mapsto q(u) = [1,\hat{x}_1(u),\cdots, \hat{x}_{n}(u)] $$
is a universal covering of $C_{k,n}'$, with $\deck(q) = \deck(p)_k$.
\end{lemma}
\begin{proof}
Since $r_i + \hat{x}_1^k + \hat{x}_i^k = \lambda_{i-2} - p + (p-r_i) = 0$ for $i\geq 2$,  we have $q(U) \subset C'_{k,n}.$
Let $a\in C'_{k,n}$ and assume $a = [1,a_1,...,a_{n}]$, with $a_i \neq 0$ for every $i$. If $q(u) =a$, then $\hat{x}_i(u) = a_i$ for each $i$. In particular $\hat{x}_1(u)^k = -p(u) = a_1^k$. Since $p$ is surjective, there exists $u \in U$ such that $p(u) = -a_1^k$, and hence $\hat{x}_1(u) = \zeta_k^{j_1} a_1$ for some integer $j_1$. Since $a_i^k = -a_1^k - r_i$ for $1<i \leq  n$ and  
$$\hat{x}_i^k(u) = p(u) - \lambda_{i-2} =-a^k_1 - r_i = a_i^k,\quad 1<i \leq  n,$$
we have  $\hat{x}_i(u) = \zeta_k^{j_i} a_i$ with an integer $j_i$ for each $i$. Now we choose $\phi \in \deck(p)$ such that
$$L(p_i , \phi) = -j_i $$
for every $i$, we get $q(\phi(u)) = a$, and therefore $q$ is surjective. Now $q$ is a covering map because every $a\in C_{k,n}'$ has an evenly covered neighborhood since $p$ is a covering map.  Finally, if $u,v \in q^{-1}(a)$ with $a \in C'_{k,n}$, then $v = \phi(u)$ for some $\phi \in \deck(p)$. Now, as
$$q(u) =  [1,\hat{x}_{1}(u) , ... ,\hat{x}_{n}(u) ]  = [1, \zeta_k^{L(p_1, \phi)}\hat{x}_{1}(u) , ... ,\zeta_k^{L(p_{n}, \phi)}\hat{x}_{n}(u) ] = q(v),$$
 we have $\phi \in \deck(p)_k$. 
\end{proof}

From the previous two lemmas, we have
\begin{lemma}
The map $\Psi$ gives an isomorphism
$$\deck(p)/\deck(q) \cong\bZ_k^n .$$
\end{lemma}

We denote by $\phi_1,\dotso,\phi_{n}$ the $n$ generators of $\deck(p)$ with
$$L(p_i, \phi_j) = \delta_{ij}, \quad 1\leq i,j \leq n.$$ 

\begin{lemma}  $\deck(q)$ is generated by
$$\phi_i^k \text{ for each } 1\leq i \leq n \quad \text{and} \quad [\deck(p) , \deck(p)], $$
where $[\deck(p) , \deck(p)]$ is the commutator subgroup of $\deck(p)$.
\end{lemma}

\begin{proof}
Since $\deck(p)/\deck(q)$ is Abelian, we have that $[\deck(p) , \deck(p)] \trianglelefteq \deck(q) $.
We also know that the free generators $\phi_1,...,\phi_n$ of $\deck(p)$ correspond to the canonical basis of $\bZ_k^n$ by $\Psi$, hence $\Psi(\phi_i^k) = 0$ for each $i$. If $K\leq \deck(p)$ is the subgroup generated by each $\phi_i^k$ and $[\deck(p),\deck(p)]$, then we have 
$$ \deck(p)/K \cong \bZ_k^n.$$
So we must have $\deck(q) = K$. 
\end{proof}

Recall that $\pi_1(C_{k,n}') = \deck(q)$ by Lemma \ref{deckk}, so we have

\begin{theorem} \label{teoremaHpinch} The first homology group of $C_{k,n}'$, namely
$$H_1(C_{k,n}' , \bZ) \cong \frac{\deck(q)}{[\deck(q) , \deck(q)]}, $$
is generated by the classes of the elements
$$\phi_i^{k},  \quad 1\leq i \leq n$$ 
and
$$\left(\prod_{d=1}^{n} \phi_d^{g_d} \right) [\phi_j , \phi_{l}] \left(\prod_{d=1}^{n} \phi_d^{g_d} \right)^{-1}, $$
with  integers $ 1 \leq j < l \leq n$ and $0\leq g_d \leq k-1$.
\end{theorem}

\begin{proof} Since $\deck(q)$ is generated by $\phi_i^k$ and $[\deck(p) , \deck(p)]$, it is generated by
$$\phi_i^k, \quad 1\leq i \leq n,$$
and
$$\gamma [\phi_j, \phi_l]\gamma^{-1},$$
\noindent with $\gamma \in \deck(p)$ and  $ 1 \leq j < l \leq n$. We have $\deck(p)/\deck(q) = \bZ_k^n$, so $ \{ \prod_{d=1}^{n} \phi_d^{g_d}\}_{0 \leq g_d \leq k-1}$ is a set of representatives such that every $\gamma \in \deck(p)$ lies in $\deck(q)\rho$ for precisely one $\rho$ from this set.

\noindent Choosing the representative $\rho \in  \{ \prod_{d=1}^{n} \phi_d^{g_d}\}_{0 \leq g_d \leq k-1}$, we have $\gamma =\sigma \rho$ with $\sigma \in \deck(q)$, and
$$\gamma [\phi_j,\phi_l]\gamma^{-1} = \sigma (\rho [\phi_j,\phi_l] \rho^{-1}) \sigma^{-1} $$
as a product of elements in $\deck(q)$. Quotienting by $[\deck(q) , \deck(q) ]$ the product commutes, and the $\sigma$'s cancel.
\end{proof}

Since the inclusion map $\iota:C_{k,n}' \hookrightarrow C_{k,n}$ induces a surjective homomorphism between the homology groups, we have

\begin{corollary}
The images of the generating set of $H_1(C_{k,n}', \bZ)$ under the homomorphism induced by the inclusion $ \iota:C_{k,n}' \hookrightarrow C_{k,n}$ forms a generating set for $H_1(C_{k,n} , \bZ )$.
\end{corollary}

We summarize the maps used in the following diagram.

\[
\xymatrix{
U \ar@/_2pc/[dd]_{p} \ar[d]^{q}\\
C_{k,n}' \ar[d]^{\pi} \ar@{^{(}->}[r]^{\iota} & C_{k,n} \ar[d]^{\pi} \\ \hat{\bC}-R\cup\{\infty\}  \ar@{^{(}->}[r] & \hat{\bC} 
}
\]
\subsection{Computing periods}
Let $\phi \in \deck(p)$ and fix $u \in U$. We denote by $l_{\phi}$ a curve from $u$ to $\phi(u)$ on $U$. So a generating set for $H_1(C_{k,n} , \bZ)$ are the homology classes of the curves $\iota \circ q (l_{\phi})$ for each $\phi$ of the form

$$\phi_i^{k},\quad \left(\prod_{d=1}^{n} \phi_d^{g_d} \right) [\phi_j , \phi_{l}] \left(\prod_{d=1}^{n} \phi_d^{g_d} \right)^{-1}$$
for $1\leq i \leq n$,  $1 \leq j < l \leq n$, and $0\leq g_d \leq k-1$. Thus, to find an explicit generating set for the period lattice of $C_{k,n}$ we need to calculate

$$\int_{\iota \circ q \circ l_{\phi}} \theta_{\alpha_1,  \dotso ,\alpha_{n}}   =\int_{l_{\phi}} q^* \theta_{\alpha_1,  \dotso ,\alpha_{n}} =  \int_{u}^{\phi(u)}q^* \theta_{\alpha_1,  \dotso ,\alpha_{n}} .$$

\begin{lemma}\label{eigen} We have the following relations between the induced pullbacks of the generators of $H^{1,0}(C_{k,n})$:
\begin{equation}\label{eq:pback1}
q^* \theta_{\alpha_1,  \dotso ,\alpha_{n}}  =  \frac{\hat{x}_1^{\alpha_1}d\hat{x}_1}{\hat{x}_2^{\alpha_{2}}\cdots \hat{x}_{n}^{\alpha_{n}}},
\end{equation}

\begin{equation} \label{eq:pback2}
\phi^* q^* \theta_{\alpha_1,  \dotso ,\alpha_{n}}  = \zeta_k^{(\alpha_1+1)L(p_1,\phi) - \sum_{d=2}^{n} \alpha_{d} L( p_d , \phi ) \alpha_d}  q^* \theta_{\alpha_1,  \dotso ,\alpha_{n}},
\end{equation}
for each $\phi \in \deck(p)$. In particular, $ q^* \theta_{\alpha_1,  \dotso ,\alpha_{n}}$ is an eigenvector for each $\phi^* \in \deck(p)$.
\end{lemma}

\proof{ The first result follows from the observation that $\hat{x}_i = y_i \circ q$ for each $i$. The second follows from Lemma \ref{powerlog} in Section \ref{logarith}. \qed
}

\vspace{10pt}
We denote by $M_i$ the value $(\alpha_1+1)L(p_1,\phi_i) -  \sum_{d=2}^{n} \alpha_d L( p_d ,\phi_i)$. We observe that

$$M_i = \left\{ \begin{array}{ccc} \alpha_1+1 & &i=1 \\ -\alpha_{i+1} & &2 \leq i \leq n. \end{array} \right.  $$

\noindent Moreover, since $U$ is a simply connected domain, we have for $\phi, \psi \in \deck(p)$ and $\omega \in H^{1,0}(U)$ the relation
$$\int_{u}^{\phi \circ \psi (u)} \omega = \int_{u}^{\phi (u)}\omega + \int_{u}^{\psi (u)} \phi^{*} \omega .$$

\begin{lemma}
For each $\phi_i^k$ with $i \in \{1,...,n\}$ we have
$$\int_{\iota \circ q \circ l_{\phi_i^k}} \theta_{\alpha_1 , ... ,\alpha_{n}} = 0.$$
\end{lemma}
\proof{If $M_i$ is non zero, then from Equation \eqref{eq:pback2} of Lemma \ref{eigen} we obtain
\begin{align*} 
\int_{\iota \circ q \circ l_{\phi_i^k}} \theta_{\alpha_1 , ... ,\alpha_{n}}  &= \int_{u}^{\phi_i^k u} q^{*} \theta_{\alpha_1, ... ,\alpha_{n}}\\
 &= \sum_{m=1}^{k} \int_{u}^{\phi_i u} (\phi_{i}^{m-1})^{*}q^{*} \theta_{\alpha_1,  \dotso ,\alpha_{n}}\\
&= \int_{u}^{\phi_i u} q^{*} \theta_{\alpha_1,  \dotso ,\alpha_{n}} \sum_{m=1}^{k} \zeta_k^{ (m-1) M_i} \\
&= 0.
\end{align*}
In the case where $M_i = 0$, the differential form  $\theta_{\alpha_1 , ... ,\alpha_{n}}$ is holomorphic on the interior of the loop $\iota \circ q \circ l_{\phi_i^k}$, so that the integral vanishes.
\qed}\\

We conclude that the homology class of each $\phi_i^k$ is null in $H(C_{k,n}, \bZ)$, which reduces the problem to computing the integrals over
 $$\left(\prod_{d=1}^{n} \phi_d^{g_d} \right) [\phi_j , \phi_{l}] \left(\prod_{d=1}^{n} \phi_d^{g_d} \right)^{-1}. $$
 
 \begin{lemma}\label{sumasuma}
For each $\sigma= \rho [\phi_j , \phi_l ] \rho^{-1}$ with $\rho = \prod_{d=1}^{n} \phi_d^{g_d}$ we have
$$\int_{\iota \circ q \circ l_{\sigma}} \theta_{\alpha_1,  \dotso ,\alpha_{n}}   = \zeta_k^{ \sum_{d=1}^{n} g_d M_d} \int_{\iota \circ q \circ l_{[\phi_j , \phi_l]} } \theta_{\alpha_1,  \dotso ,\alpha_{n}}. $$
\end{lemma}
\proof{From Lemma \ref{eigen} and the observation that $[\phi_j,\phi_l] \in \deck(q)$ leaves $q^{*}  \theta_{\alpha_1,  \dotso ,\alpha_{n}}$ invariant, it follows  that
\begin{align*}
\int_{\iota \circ q \circ l_{\sigma}}  \theta_{\alpha_1,  \dotso ,\alpha_{n}} &= \int_{u}^{\rho[\phi_j,\phi_l]   \rho^{-1} u} q^{*}  \theta_{\alpha_1,  \dotso ,\alpha_{n}}\\
&=\int_{\rho \rho^{-1} u}^{\rho[\phi_j,\phi_l]  \rho^{-1} u} q^{*}  \theta_{\alpha_1,  \dotso ,\alpha_{n}}\\
&=\int_{\rho^{-1} u}^{ [\phi_j,\phi_l]  \rho^{-1} u} \rho^{*}q^{*}  \theta_{\alpha_1,  \dotso ,\alpha_{n}}\\
&= \zeta_k^{ \sum_{d=1}^{n} g_d M_d}  \int_{\rho^{-1} u}^{[\phi_j,\phi_l]  \rho^{-1} u} q^{*}  \theta_{\alpha_1,  \dotso ,\alpha_{n}}\\
&= \zeta_k^{ \sum_{d=1}^{n} g_d M_d} \left(  \int_{\rho^{-1} u}^{[\phi_j,\phi_l]  u}  + \int_{ [\phi_j,\phi_l]   u}^{[\phi_j,\phi_l]  \rho^{-1} u}  \right)q^{*}  \theta_{\alpha_1,  \dotso ,\alpha_{n}}\\
&= \zeta_k^{ \sum_{d=1}^{n} g_d M_d}  \int_{u}^{[\phi_j,\phi_l] u} q^{*}  \theta_{\alpha_1,  \dotso ,\alpha_{n}}.
\end{align*}

\qed
}

\begin{lemma}\label{separa}
For each $j,l \in \{1,...,n \}$  we have 
\begin{align*}
\int_{\iota \circ q \circ l_{[\phi_j,\phi_l]}} \theta_{\alpha_1,  \dotso ,\alpha_{n}} &=  (1 - \zeta_k^{M_l}) \int_{u}^{\phi_j u}q^{*}  \theta_{\alpha_1,  \dotso ,\alpha_{n}}\\  &- (1- \zeta_k^{M_j}) \int_{u}^{\phi_l u}q^{*}  \theta_{\alpha_1,  \dotso ,\alpha_{n}}.
\end{align*}
\end{lemma}

\proof{Again from Lemma \ref{eigen} we have the relations
\begin{align*}
\int_{u}^{[\phi_j, \phi_l] u} q^{*} \theta_{\alpha_1,  \dotso ,\alpha_{n}}&= \int_{u}^{\phi_j u} q^{*} \theta_{\alpha_1,  \dotso ,\alpha_{n}}+ \int_{\phi_j u}^{ \phi_j \phi _l \phi_j^{-1} \phi_l^{-1} u} q^{*} \theta_{\alpha_1,  \dotso ,\alpha_{n}}\\
&= \int_{u}^{\phi_j u} q^{*} \theta_{\alpha_1,  \dotso ,\alpha_{n}}+ \int_{u}^{ \phi _l \phi_j^{-1} \phi_l^{-1} u} \phi_j^{*}q^{*} \theta_{\alpha_1,  \dotso ,\alpha_{n}}\\
 &= \int_{u}^{\phi_j u} q^{*}\theta_{\alpha_1,  \dotso ,\alpha_{n}}+ \zeta_k^{M_j} \int_{u}^{ \phi _l \phi_j^{-1} \phi_l^{-1} u} q^{*} \theta_{\alpha_1,  \dotso ,\alpha_{n}}
\end{align*}
and
\begin{align*}
\int_{u}^{ \phi _l \phi_j^{-1} \phi_l^{-1} u} q^{*} \theta_{\alpha_1,  \dotso ,\alpha_{n}}&= \int_{u}^{\phi_l u} q^{*} \theta_{\alpha_1,  \dotso ,\alpha_{n}}+ \int_{\phi_l u}^{\phi_l \phi_j^{-1} \phi_l^{-1} u} q^{*} \theta_{\alpha_1,  \dotso ,\alpha_{n}}\\
&= \int_{u}^{\phi_l u} q^{*} \theta_{\alpha_1,  \dotso ,\alpha_{n}}+ \int_{ u}^{ \phi_j^{-1} \phi_l^{-1} u} \phi_l^{*}q^{*} \theta_{\alpha_1,  \dotso ,\alpha_{n}}\\
&= \int_{u}^{\phi_l u} q^{*} \theta_{\alpha_1,  \dotso ,\alpha_{n}}+ \zeta_k^{M_l}\int_{ u}^{ \phi_j^{-1} \phi_l^{-1} u} q^{*} \theta_{\alpha_1,  \dotso ,\alpha_{n}}.
\end{align*}
Doing the same for the integral from $u$ to $\phi_j^{-1}\phi_l^{-1}u$ and observing that
$$\int_{ u}^{  \phi_l^{-1} u} (\phi_j^{-1})^{*}q^{*} \theta_{\alpha_1,  \dotso ,\alpha_{n}}= \zeta_k^{-M_j}\int_{ u}^{  \phi_l^{-1} u} q^{*} \theta_{\alpha_1,  \dotso ,\alpha_{n}}, $$
$$\int_{ u}^{  \phi_j^{-1} u} q^{*} \theta_{\alpha_1,  \dotso ,\alpha_{n}}= -\zeta_k^{M_j}  \int_{ u}^{  \phi_j u} q^{*} \theta_{\alpha_1,  \dotso ,\alpha_{n}}  $$
and
$$\int_{ u}^{  \phi_l^{-1} u} q^{*} \theta_{\alpha_1,  \dotso ,\alpha_{n}}= -\zeta_k^{M_l}  \int_{ u}^{  \phi_l u} q^{*} \theta_{\alpha_1,  \dotso ,\alpha_{n}},  $$
the result follows.\\
\qed}
 
 Finally we reduced the problem to computing
 $$ \int_{u}^{\phi_i u}q^{*} \theta_{\alpha_1,  \dotso ,\alpha_{n}}\quad \text{ for } 1\leq i \leq n . $$
 
 \begin{lemma}\label{expliciti}
Fix $u\in U$ and let $z_0 = p(u) \in \bC\setminus R$ with 
$$R = \{ r_1 = 0, r_2 = 1 , r_3 = \lambda_1, \dotso ,r_n = \lambda_{n-2} \}.$$ Then for each $i \in \{1,...,n \}$ we have
$$ \int_{u}^{\phi_i u}q^{*} \theta_{\alpha_1,  \dotso ,\alpha_{n}}=  -\frac{1}{k}(1- \zeta_k^{M_i}) \int_{z_0}^{r_i} (-w)^{\frac{\alpha_1+1}{k} -1}\prod_{t=2}^{n}(w-r_t)^{-\alpha_t/k} dw,$$
where the choice of the branch is determined by the preimage $u$ of $z_0$.
\end{lemma}

\proof{
Since $-\hat{x}_1^k = p $, making a change of variable $w =  -\hat{x}_1^k $ we obtain
$$\int_{u}^{\phi_i u}q^{*} \theta_{\alpha_1,  \dotso ,\alpha_{n}}= -\frac{1}{k}\int_{\gamma_i}(-w)^{\frac{\alpha_1+1}{k} -1}\prod_{t=2}^{n}(w-r_t)^{-\alpha_t/k} dw,$$
where $\gamma_i$ is the projection by $p$ of the curve from $u$ to $\phi_i(u)$, i.e., $\gamma_i$ is an element of $\pi_1(\bC- R,z_0)$ that surrounds $r_i$ with index $1$ .
Choose $s_0\in [0,1)$ such that $z_0 =|z_0| e^{2\pi i s_0}$, and consider the circle with center $0$ and radius $\epsilon >0$ given by $\beta_{\epsilon}(s) = \epsilon e^{2\pi i(s+s_0)}$. Let $\beta$ be the line from $z_0$ to $z_{\epsilon} \in \beta_{\epsilon} \cap \overline{z_0,0}$. Thus $\gamma_1$ is homotopic to $\beta + \beta_{\epsilon} - e^{2\pi i}\beta$, where the factor $e^{2\pi i}$ is due to the continuation of the argument through the critic line $(-\infty, 0] $, as we see in Figure \ref{fig:bibi}.
 \begin{figure}[h]
\begin{center}
\includegraphics[scale = 0.45]{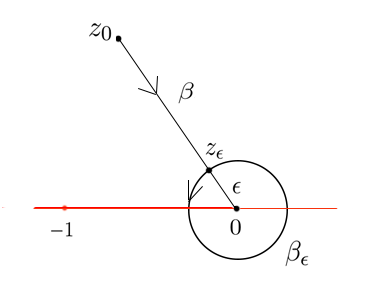}
\end{center}
\caption{}
\label{fig:bibi}
\end{figure}

\noindent Then
\begin{align*}
-k\int_{u}^{\phi_1 u} q^{*}  \theta_{\alpha_1,  \dotso ,\alpha_{n}} &= (1 - \zeta_k^{\alpha_1+1}) \int_{z_0}^{z_\epsilon} (-w)^{\frac{\alpha_1+1}{k} -1}\prod_{t=2}^{n}(w-r_t)^{-\alpha_t/k}dw\\
&+ \int_{\beta_{\epsilon}} (-w)^{\frac{\alpha_1+1}{k} -1}\prod_{t=2}^{n}(w-r_t)^{-\alpha_t/k}dw.
\end{align*}
For small $\epsilon$ the maps $(\epsilon e^{2\pi i(s+s_0)} - \lambda_t)^{-\alpha_{t}/k}$ are continuous on $[0,1]$, thus bounded. So there exists a positive constant $C$ independent of $\epsilon$ such that
\begin{align*}
&\left| \int_{\beta_{\epsilon}} (-w)^{\frac{\alpha_1+1}{k} -1}\prod_{t=2}^{n}(w-r_t)^{-\alpha_t/k} dw\right|\\
&= \left|2\pi (-\epsilon)^{\frac{\alpha_1+1}{k}}\int_{0}^{1} \frac{e^{\frac{2\pi i(\alpha_1+1)(s+ s_0)}{k}}}{ \prod_{t=2}^{n}(\epsilon e^{2\pi i(s+s_0)} - \lambda_t)^{\alpha_{t}/k}} ds \right|\\
&\leq 2\pi \epsilon^{\frac{\alpha_1+1}{k}} C.
\end{align*}

\noindent Since $\frac{\alpha_1+1}{k}, -\alpha_i/k$ are larger than $-1$ for each $i\geq 2$, in the limit $\epsilon \to 0$ we obtain
$$ \int_{u}^{\phi_1 u} q^{*} \theta_{\alpha_1,  \dotso ,\alpha_{n}} = -\frac{(1 - \zeta_k^{\alpha_1+1})}{k} \int_{z_0}^{0} (-w)^{\frac{\alpha_1+1}{k} -1}\prod_{t=2}^{n}(w-r_t)^{-\alpha_t/k} dw.$$
For $\gamma_i$ with $i \geq 2$ we apply an analogous argument.
\qed
}
\vspace{5pt}

\begin{remark}
For the integral $$\int_{z_0}^{r_i} (-w)^{\frac{\alpha_1+1}{k} -1}\prod_{t=2}^{n}(w-r_t)^{-\alpha_t/k} dw,$$
the convergence is given by the fact that $\frac{\alpha_1+1}{k}, -\alpha_j/k$ are larger than $-1$ for each $j$, so the maps $|w - r_j|^{M_j/k}$ with $i\neq j$ are well defined, continuous and
bounded when $z_0$ is in a neighborhood of $r_i$.
\end{remark}

 \begin{theorem} Let $$R = \{r_1 = 0, r_2 = 1 , r_3 = \lambda_1, \dotso ,r_n = \lambda_{n-2} \}$$
 be the set of branch points of $(C_{k,n},H_0)$ distinct of $\infty$. If we denote
$$W(R ,  \vec{\alpha})(w) :=  (-w)^{\frac{\alpha_1+1}{k} -1}\prod_{t=2}^{n}(w-r_t)^{-\alpha_t/k}  $$
for each $\vec{\alpha} = (\alpha_1,...,\alpha_{n}) \in I_{k,n}$, then the period lattice $\Lambda \cong \tau(H_1(C_{k,n}, \bZ))$ is generated by the period vectors
$$\left( \zeta_k^{ \sum_{d=1}^{n} g_d M_d}\frac{ (1 - \zeta_k^{M_j})(1-  \zeta_k^{M_l})}{k} \int_{r_j}^{r_l} W(R , \vec{\alpha}) dw  \right)_{\vec{\alpha} \in I_{k,n}}$$
for each generator $ \rho [\phi_j , \phi_l ] \rho^{-1} \in H_1(C_{k,n} , \bZ)$ with $\rho = \prod_{d=1}^{n} \phi_d^{g_d}$ and $0 \leq g_d \leq k-1$.
\end{theorem}
\proof{ From Lemmas \ref{separa} and \ref{expliciti} we obtain
 
 \begin{align*}
&\int_{u}^{[\phi_j , \phi_l] u} q^{*} \theta_{\alpha_1,  \dotso ,\alpha_{n}}\\
 &= -\frac{(1 - \zeta_k^{M_j})(1-  \zeta_k^{M_l})}{k}   \int_{r_j}^{r_l}  (-w)^{\frac{\alpha_1+1}{k} -1}\prod_{t=2}^{n}(w-r_t)^{-\alpha_t/k}dw.
\end{align*}
Thus by Lemma \ref{sumasuma} for each generator $\sigma = \rho [\phi_j , \phi_l ] \rho^{-1} \in H_1(C_{k,n} , \bZ)$ with $\rho = \prod_{d=1}^{n} \phi_d^{g_d}$ we have

\begin{align*}
&\int_{\iota \circ q \circ l_{\sigma}} \theta_{\alpha_1,  \dotso ,\alpha_{n}}\\  
&=- \zeta_k^{ \sum_{d=1}^{n} g_d M_d}\frac{ (1 - \zeta_k^{M_j})(1-  \zeta_k^{M_l})}{k} \int_{r_j}^{r_l} (-w)^{\frac{\alpha_1-1}{k} -1}\prod_{t=2}^{n}(w-r_t)^{-\alpha_t/k}dw
\end{align*}
with  $1\leq j<l \leq n$ and $0 \leq g_d \leq k-1$.
\qed}

\begin{remark}
In the case of the classical Fermat curves $C_{k,2}$ with $R = \{r_1= 0,r_2 = 1\}$, the integrals to compute are
$$ \int_{0}^{1}    \frac{(-w)^{\frac{\alpha_1+1}{k}-1 } dw}{(w-1)^{\alpha_{2}/k} } = -\eta^{\alpha_1- \alpha_{2}+1} \int_{0}^{1}  w^{\frac{\alpha_1+1}{k}-1 }(1-w)^{-\alpha_{2}/k} dw,$$
where $\eta = (-1)^{1/k}$. If we consider the Beta function
$$B(x,y) = \int_{0}^{1} t^{x-1}(1-t)^{y-1},\quad \mbox{Re}(x), \mbox{Re}(y) >0,$$
then
$$  \int_{0}^{1}    \frac{(-w)^{\frac{\alpha_1+1}{k}-1 } dw}{(w-1)^{\alpha_{1}/k} } =- \eta^{\alpha_1 - \alpha_{2}+1}B\left(\frac{\alpha_1+1}{k} , 1 - \frac{\alpha_{2}}{k} \right),$$
which yields a result similar to that of Rohrlich in \cite{GR} for the standard Fermat curve $X^k + Y^k = Z^k$. 
In the case of the generalized Fermat curve we need to compute 
$$ \int_{\lambda_j}^{\lambda_l}    \frac{(-w)^{\frac{\alpha_1+1}{k}-1 } dw}{(w-1)^{\alpha_{2}/k} (w-\lambda_1)^{\alpha_{3}/k}\cdots (w - \lambda_{n-2})^{\alpha_{n}/k}},$$
which we can view as a natural generalization of the Beta function.
\end{remark}

\section*{Acknowledgements}
These results was obtained in my Master degree where my advisor was Mariela Carvacho. I thank Rub\'en Hidalgo for his helpful suggestions and comments. I would also like to thank the referee for her/his valuable comments which improved the
presentation of the paper and saved us from several mistakes. This work was partially supported by Anillo PIA ACT1415 and Proyecto Interno USM 116.12.2.





\end{document}